\documentclass[12pt]{amsart}

\RequirePackage[nonfrench]{kotex}
\usepackage[utf8]{inputenc}
\usepackage{kotex}
\usepackage{epsfig}

\usepackage{amsfonts}
\usepackage{amssymb}
\usepackage{amsmath}
\usepackage{amsthm}
\usepackage{amscd}
\usepackage{amstext}
\usepackage{latexsym}
\usepackage{array}
\usepackage{graphicx, subcaption}
\usepackage{todonotes}

\usepackage{enumerate, enumitem}

\usepackage{hyperref}
\newcommand{\ZZ}{\mathbf{Z}}
\newcommand{\CC}{\mathbf{C}}
\newcommand{\QQ}{\mathbf{Q}}
\newcommand{\RR}{\mathbf{R}}

\newcommand{\JJ}{\mathcal{J}}

\newcommand{\el}{\ell}
\newcommand{\qa}{\quad}
\newcommand{\vp}{\varphi}

\newcommand{\Ve}{\Vert  }

\newcommand{\bfone}{\mathbf{1}}
\newcommand{\asym}{\mathrm{Asymp}}
\newcommand{\dasym}{\mathrm{Asymp}'}

\newcommand{\noi}{\noindent}
\providecommand{\inp}[1]{\langle#1\rangle}

\theoremstyle{plain}

\newtheorem{theorem}{Theorem}[section]

\newtheorem{lemma}[theorem]{Lemma}

\newtheorem{definition}[theorem]{Definition}

\newtheorem{example}[theorem]{\textnormal{\textbf{Example}}}

\newtheorem{Standard Process}[theorem]{Procedures}

\theoremstyle{remark}
\newtheorem{remark1}[theorem]{Remark}

\DeclareMathOperator{\jump}{Jump}
\DeclareMathOperator{\inter}{int}

\begin{document}

\title{Cluster points of jumping numbers of \\ toric plurisubharmonic functions}

\author{ Hoseob Seo }

\date{}

\maketitle

\begin{abstract}

\noindent  We show that the set of cluster points of jumping numbers of a toric plurisubharmonic function in $\CC^n$ is discrete for every $n \ge 1$. We also give a precise characterization of the set of those cluster points. These generalize a recent result of D. Kim and H. Seo from $n=2$ to arbitrary dimension. Our method is to analyze the asymptotic behaviors of Newton convex bodies associated to toric plurisubharmonic functions.

\end{abstract}

\section{Introduction}

In recent years, multiplier ideals and their jumping numbers played important roles in algebraic geometry (see \cite{D}, \cite{D11}, \cite{ELSV},  \cite{L}, \cite{S98}). They are crucial tools for measuring singularities of objects such as ideals, divisors and, more generally, plurisubharmonic  functions. A real number $c > 0$ is said to be a jumping number of a plurisubharmonic function $\vp$ at $x \in \CC^n$  if the multiplier ideal $\JJ(c \vp)$ with coefficient $c$ is strictly contained in $\JJ((c-\epsilon) \vp)$ for every  $0< \epsilon < c$ at $x$.

Our main interest lies in the possible failure of discreteness of jumping numbers when a psh function does not have analytic singularities (i.e. ones arising from ideals). In other words, a priori, there can exist  cluster points (i.e. accumulation points) of jumping numbers in the set of real numbers, due to the lack of an analogue of Hironaka resolution in such  generality.  Indeed, explicit examples were found by \cite[(5.10)]{ELSV} and later by \cite{GL}, both of which are toric psh functions in dimension $2$ \footnote{The one in \cite{ELSV} is given in terms of a graded system of ideals, which can be converted to a psh function as discussed in \cite{KS}.}.

A psh function $\vp$ is  {\it toric} if $\vp (z_1, \ldots, z_n) = \vp (|z_1|, \ldots, |z_n|)$, i.e.  their values depend only on the absolute values of coordinates. Recently, \cite{KS} generalized the above two examples to examining all toric psh functions in dimension $2$ : \cite[Theorem 5.7]{KS} determined when such a cluster point exists and then identified what are those cluster points. In this paper, we generalize those results to arbitrary dimension by establishing the following:

\begin{theorem}[Main theorem, Theorem~\ref{mainthm}]
 Let $\vp$ be a toric plurisubharmonic function on a polydisk $D(0,r) \subset \CC^n$. Then the set of cluster points of jumping numbers of $\vp$ at the origin is discrete. Moreover, the set of cluster points of jumping numbers of $\vp$ is given by
    $$
    \{ m \in \RR_+ : \text{there exists $A \in \dasym(m\vp)$ which does not intersect with $P(m\vp)$}\},
    $$
 where $P(m\vp)$ is the Newton convex body of $m \vp$ and $\dasym(m\vp)$ is defined in the next section.
\end{theorem}

 \noindent In particular, this says that there is no ``cluster point of cluster points" of jumping numbers for toric psh functions.

  In general, the existence of a cluster point of jumping numbers depends on the asymptotic behavior of the Newton convex bodies.  When the dimension is $2$ in \cite{KS}, the situation was quite simple since there were only two axis lines that can be asymptotic to a given Newton convex body. However in arbitrary dimension, the asymptotic behaviors of the Newton convex bodies become much more complicated. In Section~\ref{PropertiesNewtonConvex}, we present some definitions and properties which are useful to investigate the asymptotic behaviors of the Newton convex bodies. In Section~\ref{mainsection}, we prove the main theorem, Theorem \ref{mainthm}, using Lemma~\ref{ClusterToAffine}. \\

\noi \textbf{Acknowledgements.}

The author would like to thank Dano Kim for his helpful and valuable comments.
This research was supported by BK21 PLUS SNU Mathematical Sciences Division.

\section{Newton convex bodies of plurisubharmonic functions} \label{PropertiesNewtonConvex}

In \cite{Ho}, \cite{G}, \cite{R}, multiplier ideals of monomial ideals and toric plurisubharmonic (psh, for short) functions are characterized by their Newton convex bodies in $\RR^n_+$. In this paper, we use the following notations:
\begin{enumerate}[leftmargin=*]
\item $\ZZ^n_+$, $\QQ^n_+$ and $\RR^n_+$ stand for the set of $n$-tuples of nonnegative integers, rational numbers and real numbers respectively.
\item For two $n$-tuples $(x_1,\ldots, x_n), (y_1, \ldots, y_n) \in \RR^n$, write $(x_1,\ldots, x_n) < (y_1, \ldots, y_n)$ (resp. $(x_1,\ldots, x_n) \le (y_1, \ldots, y_n)$) for the case when $x_i < y_i$ (resp. $x_i \le y_i$) for all $i$.
\end{enumerate}

To make clear, we recall the definition of affine subspaces in $\RR^n$, which will be frequently used in this paper. We refer to \cite{H} for general properties of affine subspaces.

\begin{definition}[\cite{H}]
A subset $A$ of $\RR^n$ is called a $k$-dimensional affine subspace if
$$
A + (-x_0) = \{ x - x_0 : x \in A\}
$$
is a $k$-dimensional linear subspace of $\RR^n$ for some (and thus arbitrary) $x_0 \in A$. 
\end{definition}

Since one of the most important properties of Newton convex bodies $P$ is the fact that $P + \RR^n_+$ is contained in $P$, we develop some useful properties of convex subsets $P$ in $\RR^n_+$ satisfying $P + \RR^n_+ \subseteq P$. To investigate the asymptotic behavior of the boundary of $P$ far from the origin, we consider the following definition.

\begin{definition} \label{asymptotic}
Let $P$ be a convex subset of $\RR^n_+$ satisfying $P + \RR^n_+ \subseteq P$. A $k$-dimensional affine subspace $A$ is said to be {\it asymptotic to $P$} if $A \cap \inter P = \emptyset$ and $d(A,P) = 0$, where $\inter P$ is the interior of $P$ in $\RR^n$ and $d(X,Y) = \inf \{ \Ve x-y \Ve : x \in X, y \in Y \}$ for any subsets $X,Y \subseteq \RR^n$.
\end{definition}

For each cluster point of jumping numbers $\lambda$,  we can find an affine subspace $A$ associated to $\lambda$ (see Lemma~\ref{ClusterToAffine}). To prove the discreteness of cluster points of jumping numbers, it should be possible to choose $A$ to be of maximal dimension among affine subspaces associated to $\lambda$.

\begin{definition} \label{DefAffine1}
Let $\asym_{n-1}(P)$ be the set of all $(n-1)$-dimensional affine subspaces $A$ in $\RR^n$ satisfying
\begin{enumerate}[leftmargin=*]
    \item $A$ is given by an equation $x_i = a$ for some $i \in \{1,\ldots, n\}$ and some positive integer $a$,
    \item $A$ is asymptotic to $P$.
\end{enumerate}
\end{definition}

\begin{definition} \label{DefAffine2}
For $1 \le k \le n-2$, define $\asym_{k}(P)$ inductively as the set of all $k$-dimensional affine subspaces $A$ satisfying
\begin{enumerate}[leftmargin=*]
    \item[(1)] $A$ is given by equations $x_{i_1} = a_1, \ldots, x_{i_{n-k}} = a_{n-k}$ where $i_1, \ldots, i_{n-k}$ are distinct elements of $\{1,\ldots, n\}$ and $a_1,\ldots, a_{n-k}$ are positive integers,
    \item[(2)] $A$ is asymptotic to $P$,
    \item[(3)] $A$ is not contained in any affine subspaces in $\asym_m(P)$ if $m > k$.
\end{enumerate}

 We also let $\displaystyle \asym(P) := \bigcup_{k=1}^{n-1} \asym_k (P)$.
\end{definition}

We will say that an affine subspace satisfies the condition ($\star$) for $P$ if it satisfies only the conditions (1) and (2) for $P$ in Definition~\ref{DefAffine2}: \\

\begin{enumerate}[leftmargin=*]
  \item[($\star$)] \parbox{\dimexpr\textwidth-\leftmargin-\itemindent}{ 
\begin{enumerate} 
    \item[(1)] $A$ is given by equations $x_{i_1} = a_1, \ldots, x_{i_{n-k}} = a_{n-k}$, where $i_1, \ldots, i_{n-k}$ are distinct elements of $\{1,\ldots, n\}$ and $a_1,\ldots, a_{n-k}$ are positive integers,
    \item[(2)] $A$ is asymptotic to $P$.
\end{enumerate}
  }
\end{enumerate}

\begin{remark1}
Note that in the above definitions $a$ and $a_j$'s are positive integers. Therefore if a convex subset $P$ is given by
$$
P = \{ (x,y,z) \in \RR^3 : xy=1 \text{ and } x,z \ge 0 \},
$$
then the affine subspaces $x=0$ and $y=0$ are asymptotic to $P$ but they does not belong to $\asym_2 (P)$.
\end{remark1}

Additionally, in order to describe the set of cluster points of jumping numbers, we have to take an affine subspace whose dimension is minimal among affine subspaces associated to a cluster point of jumping numbers. For this purpose, define $\dasym_k(P)$ as the set of all $k$-dimensional affine subspaces in $\RR^n$ satisfying the condition $(\star)$. Similarly as above, let  $$\dasym(P) := \bigcup_{k=1}^{n-1} \dasym_k(P) .$$

\begin{example} \label{3dimexample}
Let $Q$ be a subset of $\RR^2_+$ given by
$$
Q = \left\{ (x,y) \in \RR^2 : x>1, y>1 \text{ and } y \ge \frac{1}{x-1} + 1 \right\}.
$$
For each $\epsilon > 0$, we denote by $Q_\epsilon$ the intersection of $(1-\epsilon)Q$ and $\bfone + \RR^2_+$.
Now let $P$ be the set in $\RR^3_+$ given by
$$
P = \bigcup_{0 \le \epsilon \le 1} \{ (x,y,1+\epsilon) \in \RR^3_+ : (x,y) \in Q_\epsilon \} \cup ((1,1,2) + \RR^3_+).
$$
Then $P$ is a convex subset in $\RR^3_+$ satisfying $P + \RR^3_+ \subseteq P$. The affine plane given by $x=1$ belongs to $\asym(P)$.  On the other hand, two affine line $\{(x,1,1): x \in \RR \}$ and $\{(1,y,1) : y \in \RR \}$ belong to $\dasym(P)$ but not to $\asym(P)$.
\end{example}

\begin{lemma} \label{LargestAsymp}
Let $P$ be a convex subset of $\RR^n_+$  satisfying $P + \RR^n_+ \subseteq P$ and $A$ an affine subspace asymptotic to $P$ on which at least one of coordinates of $\RR^n$ is constant. If $A'$ is the affine subspace given by coordinates which are constant on $A$, then $A'$ is asymptotic to $P$.
\end{lemma}
\begin{proof}
If $A'$ is equal to $A$, the conclusion is trivial. Assume that $A'$ strictly contains $A$. Furthermore, after rearranging the order of coordinates, we may assume that $A'$ is given by equations
$$
x_1 = a_1, \ldots, x_k = a_k.
$$
To show that $A'$ is asymptotic to $P$, assume on the contrary that $A'$ meets the interior of $P$. Let $(x',x'')$ be a point in the interior of $P$ such that $x' \in \RR^k$, $x'' \in \RR^{n-k}$ and $x' < a:= (a_1,\ldots, a_k)$. If $x''$ is chosen so that $|x''|$ is sufficiently large and $(a, x'')$ lies on $A$, then $(a,x'')$ is contained in $(x',x'') + \RR^n_+$, which implies that $(a, x'')$ is contained in the interior of $P$. This contradicts the fact that $A$ does not meet the interior of $P$. Note that such $x''$ can be taken because $A'$ strictly contains $A$.
\end{proof}

For a toric psh function $\vp$ on a polydisk $D(0,1) \subseteq \CC^n$, it is well-known that a function
$$
f(x_1, \ldots, x_n) := \vp (e^{x_1}, \ldots, e^{x_n})
$$
is a convex function on $\RR^n_-$ which is increasing in each variable, where $\RR^n_-$ is the set of n-tuples of non-positive real numbers. Then the Newton convex body of $\vp$, denoted by $P(\vp)$, is defined to be the domain of the Legendre transformation (or the convex conjugate) of the associated convex function of $\vp$ (see \cite[\S 2.2]{H} for more details). More explicitly, $P(\vp)$ is defined by
$$
P(\vp) = \{ \lambda \in \RR^n : \textstyle \sup_{x \in \RR^n_-}\{ \inp{\lambda, x} - f(x) \} < +\infty \}.
$$
Note that $P(\vp)$ is contained in $\RR^n_+$ and satisfies $P(\vp) + \RR^n_+ \subseteq P(\vp)$.

In this paper, we are concerned with the case when a convex subset $P$ is the Newton convex body $P(\vp)$ of a toric psh function $\varphi$. For simplicity of notation, we will write 
\begin{align*}
    \asym_k(P(\varphi)) = \asym_k(\varphi) \text{ (resp. $\dasym_k(P(\vp)) = \dasym_k(\vp)$),}  \\
    \asym(P(\varphi)) = \asym(\varphi) \text{ (resp. $\dasym(P(\vp)) = \dasym(\vp)$)}.    
\end{align*}

\section{Cluster points of jumping numbers of plurisubharmonic functions} \label{mainsection}

 First of all, since we are especially interested in the stalk of a multiplier ideal sheaf at the origin $0 \in \CC^n$, we rephrase the following theorem which was independently proved by Guenancia and Rashkovskii.
\begin{theorem}[{\cite[Theorem~1.13]{G}}, {\cite[Proposition~3.1]{R}}] \label{GueRash}
Let $\varphi$ be a toric plurisubharmonic function on $D(0,r) \subset \CC^n$. The stalk of the multiplier ideal sheaf $\JJ(\varphi)_0$ at the origin $0 \in \CC^n$ is a monomial ideal. Furthermore, a monomial $z^\alpha = z_1^{\alpha_1} \cdots z_n^{\alpha_n}$ is contained in $\JJ(\varphi)_0$ if and only if
$$
\alpha + \mathbf{1} \in \inter P(\varphi)
$$
where $\mathbf{1} = (1,\ldots, 1) \in \ZZ^n$.
\end{theorem}

Before we prove the main theorem, we shall introduce the following lemma. This lemma and an argument used in its proof are crucial for the proof of Theorem~\ref{mainthm}.

\begin{lemma} \label{ClusterToAffine}
Let $\lambda$ be a cluster point of $\jump(\vp)_o$ and $\{\lambda_j\}_{j=1}^{\infty}$ an increasing sequence of jumping numbers converging to $\lambda$. Then one can find an affine subspace $A_\lambda$ in $\asym(\lambda \varphi)$, a subsequence $\{\lambda_{j_k}\}_{k=1}^{\infty}$ of $\{\lambda_j\}$ and a sequence $\{x_k\}_{k=1}^{\infty}$ in $(A_\lambda + (-\bfone)) \cap \ZZ^n_+$ such that
$$
x_k + \mathbf{1} \in \inter P(\lambda_{j_k} \varphi) \quad \text{but} \quad x_k + \mathbf{1} \notin \inter P(\lambda_{j_k+1} \varphi).
$$
\end{lemma}

\begin{proof}
For each $j \ge 1$, one can find a monomial $z^{\alpha_j} \in \JJ(\lambda_j \varphi)$ which is not contained in $\JJ(\lambda_{j+1} \varphi)$. By Theorem~\ref{GueRash}, it gives a sequence $\{\alpha_j = (\alpha_{j,1}, \ldots, \alpha_{j,n})\}_{j=1}^{\infty}$ in $\ZZ^n_+$ satisfying $\alpha_j + \bfone \in (\inter P(\lambda_j \varphi)) \backslash (\inter P(\lambda_{j+1} \varphi) )$. If the sequence $\{\alpha_j\}$ is strictly increasing, that is $\alpha_1 < \alpha_2 < \cdots$, then we have $\alpha_j \in \inter P(\lambda \varphi)$ for sufficiently large $j$. This cannot happen. With this fact since all $\alpha_j$ are contained in $\ZZ^n_+$, one can find a subsequence $\{ x_k \}_{k=1}^{\infty}$ of $\{\alpha_j\}$ in which some coordinate of each element is constant. Let $\lambda_{j_k}$ be the jumping number of $\varphi$ corresponding to $x_k$ and let $A$ be the affine hull of the sequence $\{x_k + \bfone \}$. Then $A$ is asymptotic to $P(\lambda \vp)$. By Lemma~\ref{LargestAsymp}, the affine subspace $A'$ given by coordinates which are fixed on $A$ is asymptotic to $P(\lambda \vp)$. Let $A_\lambda$ be an affine subspace which is of maximal dimension among all affine subspaces containing $A'$ and asymptotic to $P(\lambda \vp)$. Then $\{ \lambda_{j_k} \}$, $\{x_k\}$ and $A_\lambda$ satisfy the statement.
\end{proof}

This lemma allows us to define a function $\Phi$ from the set of all cluster points of jumping numbers of $\vp$ to $\bigcup_{m > 0} \asym(m \vp)$ by $\Phi (\lambda) = A_\lambda$.

\begin{theorem}[Main theorem] \label{mainthm}
    Let $\vp$ be a toric psh function on a polydisk $D(0,1) \subseteq \CC^n$. Then the set of cluster points of jumping numbers of $\vp$ at the origin is discrete. Moreover, the set of cluster points of jumping numbers of $\vp$ is given by
    $$
    \{ m \in \RR_+ : \text{there exists $A \in \dasym(m\vp)$ which is does not intersect with $P(m\vp)$}\}.
    $$
\end{theorem}

\begin{remark1}
Note that Theorem~\ref{mainthm} generalizes the case where $n=2$ which was proved in \cite{KS}. When the dimension is $2$, $A(m\vp)$ consists of at most two lines for each $m > 0$ and from this we deduce that the set of cluster points of jumping numbers has of the form:
$$
\{ k/m : k \in \ZZ_+, m \in S\}
$$
where $S$ is a set which consists of at most two positive real numbers.
\end{remark1}

\begin{proof}[Proof of Theorem~\ref{mainthm}]
We shall show that for each $M>0$, the set
$$
\bigcup_{0<m\le M} \asym(m\vp)
$$
has only finitely many elements. Since it is obvious that $\Phi$ is injective, we then have the first statement. Since $\asym(m\vp)$ is the union of $\asym_1(m\vp), \ldots, \asym_{n-1}(m\vp)$, it is enough to show that for each integer $k$ with $1 \le k \le n-1$, the set
$$
\bigcup_{0 < m \le M} \asym_k(m\vp)
$$
is a finite set. To do this, we shall prove that if there are infinitely many $k$-dimensional affine subspaces, each of which satisfies the condition $(\star)$ for $P(m\vp)$ with some $0 < m \le M$, then there exists an affine subspace of dimension more than $k$ satisfying the condition $(\star)$ for $P(m'\vp)$ with $0 < m' \le M$ and containing some of them.

First of all, one can find a sequence of distinct $k$-dimensional affine subspaces $\{A_i\}_{i=1}^{\infty}$ such that after rearranging the order of coordinates, each $A_i$ is given by equations
$$
x_1 = a_{i,1}, \ldots, x_{n-k} = a_{i, n-k},
$$
where all $a_{i,j}$ are positive integers, and asymptotic to $P(m\vp)$ for some $0 < m \le M$.

Now we can associate each $A_i$ to a point $\alpha_i := (a_{i,1}, \ldots, a_{i,n-k})$ in $\ZZ^{n-k}$. Note that the sequence $\{\alpha_i \}_{i=1}^{\infty}$ cannot be strictly increasing. Thus by the same argument in the proof of Lemma~\ref{ClusterToAffine} we can find a subsequence $\{ \beta_i \}_{i=1}^{\infty}$ with abuse of notation such that there exists an index $\el$ such that the $\el$-th coordinates of all $\beta_i$'s are all the same. Denote by $a_1$ the $\el$-th coordinate of $\beta_i$ and we may assume that $\el = 1$.  Then all the affine subspaces $\{B_i\}_{i=1}^{\infty}$ corresponding to $\{\beta_i\}$ are asymptotic to $P(m_0\vp)$ for some $m_0$ with $0< m_0 \le M$, where $m_0$ is independent of $i$. Now the affine subspace given by $x_1 = a_1$ either is asymptotic to $P(m_0\vp)$ or intersects with the interior of $P(m_0\vp)$. If the latter is the case, then we repeat this process to obtain a subsequence $\{\gamma_i\}_{i=1}^{\infty}$ such that the number of indices at which coordinates of all $\gamma_i$'s are constant is maximal. We may suppose that $I = \{1, \ldots, j\}\, (1 \le j < n-k)$ is the set of all indices such that for each $\el \in I$, the $\el$-th coordinates of all $\gamma_i$'s are the same as $a_\el$. Then the affine subspace given by $x_1 = a_1,\ldots ,x_j = a_j$ satisfies the condition $(\star)$ for $P(m_0\vp)$ and has dimension $n-j > k$ as desired.

For the latter, it is enough to show that if $\Phi(m) = A_m \in \asym(m\vp)$ intersects $P(m\vp)$ then we can find an affine subspace $A$ of dimension $k < \dim A_m$ which belongs to $\dasym(m\vp)$. As in Lemma~\ref{ClusterToAffine}, we have an increasing sequence of jumping numbers $\{ m_j \}$ which converges to $m$ and a sequence of points $\{ x_j \}$ in $(A_m + (-\bfone)) \cap \ZZ^n_+$ such that
$$
x_j + \bfone \in \inter P(\lambda_j \vp) \quad \text{but} \quad x_j + \bfone \notin \inter P(\lambda_{j+1} \vp).
$$
After rearranging coordinates, we may assume that $A$ is given by equations
$$
x_1 = a_1, \ldots, x_{n-k} = a_{n-k}.
$$
So we write $x_j = (a, x'_j) \in \RR^{n-k} \times \RR^k$ where $a = (a_1,\ldots, a_{n-k})$. An argument similar to the proof of the first statement shows that one can find a subsequence of $\{x'_j\}$, denoted again by $\{x'_j\}$ for simplicity, such that there exists an index $i$ at which the coordinates of all $x'_j$'s are all equal. This implies that there is an affine subspace of dimension less than $k$ which satisfies the condition $(\star)$ for $P(m\vp)$. If this affine subspace intersects $P(m\vp)$ again, we can repeat this argument to find an affine subspace of smaller dimension. This procedure should stop at least when the dimension reaches $1$. Therefore we have an affine subspace in $\dasym(m\vp)$ satisfying the condition $(\star)$.
\end{proof}

Note that $\asym(m\vp)$ is used in the proof of Theorem~\ref{mainthm} to show the discreteness of cluster points of jumping numbers while we describe the set of cluster points of jumping numbers in terms of $\dasym(m\vp)$. The following example shows the reason that $\dasym(m\vp)$ cannot be replaced by $\asym(m\vp)$ in the second statement of Theorem~\ref{mainthm}.

\begin{example}
Let $P$ be as in Example~\ref{3dimexample} and let $\vp$ be a germ of a toric Siu psh function whose Newton convex body is $P$(see \cite[\textsection 2.3]{KS} for the definition of Siu psh functions). 

As we have seen in Example~\ref{3dimexample}, the affine subspace given by $x=1$ belongs to $\asym(\vp)$ and can be chosen to be $A_1$ in Lemma~\ref{ClusterToAffine}. But in this case $\{x_k\}$ in Lemma~\ref{ClusterToAffine} can exist only on $\{(x,1,1): x \in \RR \} \cup \{(1,y,1) : y \in \RR \}$. Thus we know that there are cluster points of jumping numbers of $\vp$ which are not obtained from the set
$$
\{ m \in \RR_+ : \text{there exists $A \in \asym(m\vp)$ which does not intersect with $P(m\vp)$}\}
$$
since $A$ intersects with $P$.
\end{example}
\footnotesize

\bibliographystyle{amsplain}

\qa

\qa

\normalsize

\noi \textsc{Hoseob Seo}

\noi Department of Mathematical Sciences, Seoul National University

\noi 1 Gwanak-ro, Gwanak-gu, Seoul 08826, Republic of Korea

\noi e-mail: hskoot@snu.ac.kr

\end{document}